\documentclass[10pt,a4paper,reqno,sumlimits]{amsart}
\usepackage[numbers]{natbib}
\usepackage{amsmath}
\usepackage{amssymb}
\usepackage[mathscr]{euscript}
\usepackage{enumerate}
\usepackage{xspace}
\usepackage{color}
\usepackage{enumerate}
\usepackage{enumitem}
\usepackage{amsthm}

\begin{document}

\theoremstyle{plain}
\newtheorem{theorem}{Theorem}[section]
\newtheorem{conditionM}{Condition M.\hspace{-0.13cm}}
\newtheorem{conditionC}{Condition C.\hspace{-0.13cm}}
\newtheorem{condition}{Condition}
\newtheorem{lemma}[theorem]{Lemma}
\newtheorem{proposition}[theorem]{Proposition}
\newtheorem{corollary}[theorem]{Corollary}
\newtheorem{claim}[theorem]{Claim}
\newtheorem{definition}[theorem]{Definition}
\newtheorem{Ass}[theorem]{Assumption}
\newcommand{\q}{\mathbf{Q}}
\theoremstyle{definition}
\newtheorem{remark}[theorem]{Remark}
\newtheorem{note}[theorem]{Note}
\newtheorem{example}[theorem]{Example}
\newtheorem{assumption}[theorem]{Assumption}
\newtheorem{SA}[theorem]{Standing Assumption}
\newtheorem*{notation}{Notation}
\newtheorem*{assuL}{Assumption ($\mathbb{L}$)}
\newtheorem*{assuAC}{Assumption ($\mathbb{AC}$)}
\newtheorem*{assuEM}{Assumption ($\mathbb{EM}$)}
\newtheorem*{assuES}{Assumption ($\mathbb{ES}$)}
\newtheorem*{assuM}{Assumption ($\mathbb{M}$)}
\newtheorem*{assuMM}{Assumption ($\mathbb{M}'$)}
\newtheorem*{assuL1}{Assumption ($\mathbb{L}1$)}
\newtheorem*{assuL2}{Assumption ($\mathbb{L}2$)}
\newtheorem*{assuL3}{Assumption ($\mathbb{L}3$)}
\newtheorem{charact}[theorem]{Characterization}
\newtheorem*{IoP}{Idea of Proof}
\newcommand{\notiz}{\textup} 
\newtheorem{Case}{Case}
\newcommand{\lebesgue}{\ensuremath{\lambda\!\!\lambda}}
\renewenvironment{proof}{{\parindent 0pt \it{ Proof:}}}{\mbox{}\hfill\mbox{$\Box\hspace{-0.5mm}$}\vskip 16pt}
\newenvironment{proofthm}[1]{{\parindent 0pt \it Proof of Theorem #1:}}{\mbox{}\hfill\mbox{$\Box\hspace{-0.5mm}$}\vskip 16pt}
\newenvironment{prooflemma}[1]{{\parindent 0pt \it Proof of Lemma #1:}}{\mbox{}\hfill\mbox{$\Box\hspace{-0.5mm}$}\vskip 16pt}
\newenvironment{proofcor}[1]{{\parindent 0pt \it Proof of Corollary #1:}}{\mbox{}\hfill\mbox{$\Box\hspace{-0.5mm}$}\vskip 16pt}
\newenvironment{proofprop}[1]{{\parindent 0pt \it Proof of Proposition #1:}}{\mbox{}\hfill\mbox{$\Box\hspace{-0.5mm}$}\vskip 16pt}

\let\MID\mid
\renewcommand{\mid}{|}

\let\SETMINUS\setminus
\renewcommand{\setminus}{\backslash}

\def\stackrelboth#1#2#3{\mathrel{\mathop{#2}\limits^{#1}_{#3}}}

\renewcommand{\theequation}{\thesection.\arabic{equation}}
\numberwithin{equation}{section}

\newcommand{\pH}{\p^{\mathscr{H}^1}}
\newcommand{\1}{\mathbf{1}}
\newcommand{\of}{[\![}
\newcommand{\gs}{]\!]}
\newcommand{\cadlag}{c\`adl\`ag }
\newcommand{\BF}{B^{X, \F}(h)}
\newcommand{\CF}{C^{X, \F}}
\newcommand{\nuF}{\nu^{X, \F}}
\newcommand{\BG}{B^{X, \G}(h)}
\newcommand{\CG}{C^{X, \G}}
\newcommand{\nuG}{\nu^{X, \G}}
\newcommand{\BK}{B^{X, \mathbf{K}}(h)}
\newcommand{\CK}{C^{X, \mathbf{K}}}
\newcommand{\nuK}{\nu^{X, \mathbf{K}}}
\newcommand{\cf}{c^{X, \F}}
\newcommand{\G}{\mathsf{G}}
\newcommand{\dd}{\operatorname{d}\hspace{-0.058cm}}
\newcommand{\E}{E}
\newcommand{\p}{P}
\newcommand{\pK}{\p^{\mathscr{H}^2}}
\newcommand{\Lip}{C^{1, 1}(\mathbb{R}^d, \mathbb{R})}

\title[Stochastic Orders for PIIs]{Monotone and Convex Stochastic Orders for \\Processes with Independent Increments}

\author[D. Criens]{David Criens}
\address{D. Criens - Technical University of Munich, Department of Mathematics, Germany}
\email{david.criens@tum.de}



 \thanks{The author thanks Noam Berger, Jean Jacod and Stefan Junk for valuable and fruitful discussions. Moreover, he is grateful to Ernst Eberlein and Kathrin Glau for bringing the topic to his attention.}
\thanks{D. Criens - Technical University of Munich, Department of Mathematics, Germany, \texttt{david.criens@tum.de}.}

\frenchspacing
\pagestyle{myheadings}

\begin{abstract}
We study monotone and convex stochastic orders for processes with independent increments. 
Our contributions are twofold: First, we relate stochastic orders of the Poisson component to orders of their (generalized) L\'evy measures. 
The relation is proven using an interpolation formula for infinitely divisible laws.
Second, we derive explicit conditions on the characteristics of the processes. 
In this case, we prove the conditions via constructions of couplings. 
\end{abstract}
\date{\today}
\maketitle
\section{Introduction}



In this article we study monotone and convex stochastic orders for processes with independent increments (PIIs).
The law of a PII can be described by a deterministic triplet, called the \emph{characteristics}, which has a similar structure as a L\'evy-Khinchine triplet corresponding to a L\'evy process. 
The first characteristic represents the drift, the second characteristic encodes the Gaussian component and the third characteristic measures the frequency of jumps.
Our goal is to give conditions for stochastic orders in terms of the characteristics of PIIs.

Let us explain our main ideas. We start with the observation 
that PIIs can be decomposed into two independent parts: A quasi-left continuous PII and a sum of independent random variables which represents the fixed times of discontinuity.
By the independence of the parts, for monotone and convex stochastic orders it suffices to order both parts individually. 
The fixed times of discontinuities can be ordered by ordering each summand. Hence, our main focus lies on the quasi-left continuous parts. 

In this regard, our discussion is divided into two parts.
In the first one, we decompose the (quasi-left continuous) PIIs further into a Gaussian and a Poisson component. Again, it suffices to order each of them separately. 
In the case of the Gaussian parts, conditions for finite-dimensional stochastic orders are well-studied. Thus, we restrict our discussion to the Poisson parts, for which we show that ordering the third characteristics implies finite-dimensional stochastic orders. 
In the L\'evy case, the Poisson parts are ordered if, and only if, the L\'evy measures are ordered.
The main tool in our proof is an \emph{interpolation formula} for infinitely divisible laws in the spirit of the formulas studied in \cite{houdre2002, Houdre1998}. 

In the second part, we are interested in conditions which can be read immediately from the characteristics of the PIIs. 

For the monotone stochastic order, we first give a \emph{majorization condition}: The PIIs satisfy a drift condition, have the same Gaussian components and their jump frequencies are ordered in the sense that the negative jumps of the stochastically smaller process dominate the negative jumps of the stochastically larger process and vise versa for the positive jumps.
Instead of deducing the result from our previous results, we present an alternative proof. The main idea is to couple the processes via so-called It\^o maps, which relate L\'evy measures to a reference L\'evy measure.

The alternative proof brings additional aspects to the table: First, it shows that the conditions imply a pathwise order instead of a finite-dimensional one.
Second, the It\^o maps imply an easy sufficient and necessary condition for the monotone ordering of the third characteristics, which can be considered as a generalization of the ordering of survival functions.
Third, the proof illustrates the relation between the conditions and their intuitive interpretations via the characteristics of the processes.

We also give \emph{cut criteria}, which allow the majorization of the frequencies of jumps to change once. 
In this case, we present a third alternative proof based on another coupling, which is built using the interpretation of the characteristics. 

For the convex stochastic order, we also give a \emph{majorization condition}: The expectations and the covariance functions of the PIIs are ordered and the stochastically larger PII has a higher jump frequency than the stochastically smaller PII.
For this condition we present a short proof, which applies to all PIIs with finite first moments. 
It uses the observation that the stochastically larger PII can be decomposed in law into the stochastically smaller PII and a PII with non-negative expectation such that both are independent.
Now, as in Strassen's theorem, the convex order follows by Jensen's inequality.

Comparison results for L\'evy processes and PIIs with absolutely continuous characteristics were obtained by Bergenthum and R\"uschen\-dorf \cite{bergenthum2007, bergenthum2007b} and B\"auerle, Blatter and M\"uller \cite{Bauerle2007}.
The main idea in \cite{bergenthum2007} is to
start with two compound Poisson processes with the same jump intensity and to observe that these processes can be compared by (stochastically) ordering the jump size distribution. 
By putting mass into the origin, the case of compound Poisson processes with different jump intensities can be reduced to the case with equal jump intensities. Approximation arguments yield conditions for L\'evy processes with infinite activity. 

We show that the results obtained in \cite{bergenthum2007} for compound Poisson processes with equal jump intensity hold for more general PIIs without modifying the characteristics. 
Moreover, our explicit conditions improve several results in \cite{bergenthum2007b} by showing that parts of the conditions are not necessary.

The focus in \cite{Bauerle2007} lies on the supermodular stochastic order, which is not studied in this article. 
As our first part, the proofs are based on an interpolation formula from \cite{houdre2002}, which applies to functions in \(C^2_b\).
Since the supermodular stochastic order is generated by the supermodular functions in \(C^2_b\), the interpolation formula in \cite{houdre2002} can be applied directly.
The convex stochastic order, however, is not generated by bounded functions and we 
have to generalize the interpolation formula to Lipschitz continuous functions.


This article is structured as follows.
In Section \ref{Stochastic Order and Processes with Independent Increments} we recall the concepts of PIIs and stochastic orders. 
In Section \ref{sec: main} we state and prove our general conditions and in Section \ref{sec: M} we present our majorization conditions and cut criteria together with the corresponding coupling arguments.
In Section~\ref{sec: appl} we discuss how to generalize our conditions to semimartingales with conditionally independent increments and we give examples.

Let us end the introduction with a short remark on notation: For all non-explained notation we refer the reader to \cite[Chapters I and II]{JS}.

\section{Stochastic Orders and PIIs}\label{Stochastic Order and Processes with Independent Increments}
In this section we introduce the two main objects in this article: Processes with independent increments and stochastic orders.
We fix some \(d \in \mathbb{N}\).
\begin{definition}
	An \(\mathbb{R}^d\)-valued \cadlag adapted stochastic process \(X\) on the filtered probability space \((\Omega, \mathscr{F}, (\mathscr{F}_t)_{t \in [0, \infty)}, \p)\) is called PII, if \(X_0 = 0\) and \(X_t - X_s\) is independent of \(\mathscr{F}_s\) for all \(s \in [0, t]\) and \(t \in [0, \infty)\).
\end{definition}
Stochastic ordering is a concept depending only on probability measures. Since the law of a PII can be described by a deterministic triplet, the filtration \((\mathscr{F}_t)_{t \in [0, \infty)}\) is no active player in this article. We formalizes this:
Let \(h\) be a fixed truncation function.
As stated in \cite[Theorem II.5.2]{JS}, laws of PIIs have a one-to-one correspondence to a deterministic triplet \((B, C, \nu)\), called the \emph{characteristics}, consisting of the following:
\begin{enumerate}
	\item[\textup{(i)}]
	\(B\colon [0, \infty) \to \mathbb{R}^d\) is \cadlag with~\(B(h)_0= 0\).
	\item[\textup{(ii)}]
	\(C \colon [0, \infty) \to \mathbb{R}^d \otimes\mathbb{R}^d\) is continuous with \(C_0 = 0\), such that \(C_t - C_s\) is non-negative definite for all \(0 \leq s < t\).
	\item[\textup{(iii)}]
	\(\nu\) is a \(\sigma\)-finite measure on \(([0, \infty) \times \mathbb{R}^d, \mathscr{B}([0, \infty)) \otimes \mathscr{B}(\mathbb{R}^d))\).
\end{enumerate}
Providing an intuition, \(B\) represents to the drift, \(C\) encodes the Gaussian component and \(\nu\) encodes the Poisson component.

We define 
\begin{align*}
\mathcal{F}^m_{st} &\triangleq \{f \colon \mathbb{R}^{d\cdot m} \to \mathbb{R}, f \textup{ Borel and increasing}\},\\
\mathcal{F}^m_{cx} &\triangleq \{f\colon \mathbb{R}^{d\cdot m} \to \mathbb{R}, f \textup{ convex}\},\\
\mathcal{F}^m_{icx} &\triangleq \{f \colon \mathbb{R}^{d \cdot m} \to \mathbb{R}, f \textup{ increasing and convex}\}.
\end{align*}
Here, a function \(f\colon \mathbb{R}^{d \cdot m} \to \mathbb{R}\) is increasing if \(f(x) \leq f(y)\) whenever \(x \leq y\), which means \(x_i \leq y_i\) for all \(i \leq d \cdot m\).

For two PIIs \(X\) and \(Y\), we write \(X \preceq_\bullet Y\) if for all \(m \in \mathbb{N}\), \(0 \leq t_1 < t_2 < ... < t_m < \infty\) and \(f \in \mathcal{F}^m_\bullet\) it holds that 
\begin{align}\label{so}
E[f(X_{t_1}, ..., X_{t_m})] \leq E[f(Y_{t_1}, ..., Y_{t_m})],
\end{align}
whenever the integrals are well-defined.
Needless to say that \(X\) and \(Y\) may defined on different probability spaces and that \eqref{so} is a property of the laws of \(X\) and \(Y\).

We denote by \(\mathbb{D}\) the space of all \cadlag functions \([0, \infty)\to \mathbb{R}\) and equip it with the Skorokhod topology. 
In the one-dimensional case \(d = 1\) we also consider the following class:
\[
\mathcal{F}_{pst} \triangleq \{f \colon \mathbb{D} \to \mathbb{R}, f \text{ Borel and } f(\alpha) \leq f(\omega) \text{ if } \alpha_t \leq \omega_t \text{ for all } t \in [0, \infty)\}.
\]
We use \(p\) as an acronym for \emph{pathwise}. 
In this case, we write \(X \preceq_{pst} Y\) if \[E[f(X)] \leq E[f(Y)] \text{ for all } f \in \mathcal{F}_{pst}.\]

\section{Stochastic Orders for PIIs}\label{sec: main}
\subsection{Decomposition of PIIs}
In this section we show that a PII \(X\) with characteristics \((B^X, C^X, \nu^X)\) can be decomposed into a quasi-left continuous PII \(X^{qlc}\) and a sum of independent random variables \(X^{ftd}\) such that \(X^{qlc}\) and \(X^{ftd}\) are independent.
We assume that \(|h(x)|\1_{J^X} \star \nu^X_t < \infty\) for all \(t \in [0, \infty)\), where
\(
J^X \triangleq \{t \in [0, \infty)\colon \nu^X(\{t\}\times \mathbb{R}^d) > 0\}.
\)
Moreover, we set
\begin{align*}
\nu^{X, qlc} (\dd t \times \dd x) &\triangleq \1_{\complement J^X}(t) \nu^X(\dd t \times \dd x),\\B^{X, qlc} &\triangleq B^X - h(x)\1_{J^X} \star \nu^X.\end{align*} For a Borel function \(f\colon [0, \infty) \times \mathbb{R}^d \to [0, \infty)\) we write \[\left(f \cdot \nu^X\right)(\dd t \times \dd x) \triangleq f(t, x) \nu^X(\dd t \times \dd x).\]
Following Jacod and Shiryaev \cite{JS}, we denote the Dirac measure by \(\varepsilon\).
\begin{lemma}\label{lem: deco PII}
	The process \(X^{ftd} \triangleq x\1_{J^X} \star \mu^X\) is a.s. well-defined as a sum of independent random variables \((\Delta X_s)_{s \in J^X}\) such that 
	\begin{align}\label{law jump}
	P(\Delta X_s \in \dd x) = \nu^X(\{s\} \times \dd x) + \left(1 - \nu^X\left(\{s\}\times \mathbb{R}^d\right) \right) \varepsilon_0(\dd x).
	\end{align} 
	Moreover, 
	\(
	X^{qlc} \triangleq X - X^{ftd}
	\)
	is a quasi-left continuous PII with characteristics \((B^{X, qlc}, C^X, \nu^{X, qlc})\) and \(X^{ftd}\) and \(X^{qlc}\) are independent.
\end{lemma}
\begin{proof}
	The process \(X^{ftd}\) is well-defined since \[E\left[|h(x)| \1_{J^X}\star \mu^X_t\right] = |h(x)|\1_{J^X}\star \nu^X_t < \infty,\] by assumption, and \(|x - h(x)|\1_{J^X}\star \mu^X_t < \infty\), by the \cadlag paths of \(X\). The independence of the sequence \((\Delta X_s)_{s \in J}\) follows by the independent increments of \(X\) and the fact that independence extends to a.s. limits.
	The formula \eqref{law jump} is due to \cite[Theorem II.5.2]{JS}.
	It follows from \cite[Theorem II.5.10]{JS} that \(X^{qlc}\) is a quasi-left continuous PII with characteristics \((B^{X, qlc}, C^X, \nu^{X, qlc})\). The independence of \(X^{ftd}\) and \(X^{qlc}\) follows from \cite[Corollary 2.7, Lemma 2.8]{Kallenberg} and \cite[Theorem II.5.2]{JS}.
	Let us give a few more details on this point: By \cite[Corollary 2.7]{Kallenberg} it suffices to show that for all sequences \(0 \leq t_1 < ... < t_n< \infty\) the vectors \((X^{ftd}_{t_1}, ..., X^{ftd}_{t_n})\) and \((X^{qlc}_{t_1}, ..., X^{qlc}_{t_n})\) are independent. Moreover, by \cite[Lemma 2.6]{Kallenberg}, it even suffices to show that for all \(0 \leq s < t < \infty\) the random variables \(X^{ftd}_t - X^{ftd}_s\) and \(X^{qlc}_t -X^{qlc}_s\) are independent. We define the  (deterministic) processes
	\begin{align*}
	\widehat{B} &\triangleq \begin{pmatrix}h(x) \1_{J^X} \star \nu^X\\B^{X, qlc}\end{pmatrix},\qquad \widehat{C} \triangleq \begin{pmatrix} 0&0\\0&C^X\end{pmatrix}
	\end{align*}
	and the measure
	\begin{align*}\widehat{\nu}(\dd t \times \dd x \times \dd y) &\triangleq \1_{J^X}(t) \nu^X(\dd t \times \dd x) \varepsilon_{0} (\dd y) \\&\hspace{2cm}+ \1_{\complement J^X} (t) \nu^X(\dd t \times \dd y)\varepsilon_0 (\dd x).
	\end{align*}
	It is routine to check that the triplet \((\widehat{B}, \widehat{C}, \widehat{\nu})\) satisfies \cite[II.5.3 - II.5.5]{JS} w.r.t. the truncation function \(\hat{h}(x, y) \triangleq (h(x), h(y))\).
	Thus, using \cite[Theorem II.5.10]{JS}, it follows that the \(\mathbb{R}^{2d}\)-valued process \((X^{ftd}, X^{qlc})\) is a PII with characteristics \((\widehat{B}, \widehat{C}, \widehat{\nu})\) corresponding to the truncation function \(\hat{h}\).
	Now, an application of \cite[Theorem II.5.2]{JS} yields that for all \(u, v \in \mathbb{R}^d\) and \(0 \leq s < t < \infty\)
	\begin{align*}E&\left[\exp\left(\sqrt{-1} \left\langle u, X^{ftd}_t - X^{ftd}_s\right\rangle + \sqrt{-1} \left\langle v, X^{qlc}_t - X^{qlc}_s\right\rangle\right)\right] \\&\hspace{0.8cm}= E\left[\exp\left(\sqrt{-1}\left \langle u, X^{ftd}_t - X^{ftd}_s\right\rangle\right)\right] E\left[\exp\left( \sqrt{-1}\left\langle v, X^{qlc}_t - X^{qlc}_s\right\rangle\right)\right].\end{align*}
	Thus, the independence of \(X^{ftd}_t\) and \(X^{qlc}_t\) follows by the uniqueness theorem characteristics functions.
\end{proof}

By the independence we can consider the fixed times of discontinuity and the quasi-left continuous parts separately. 
Let \(Y\) be a second PII with characteristics \((B^Y, C^Y, \nu^Y)\) such that \(|h(x)| \1_{J^Y} \star \nu^Y_t < \infty\) for all \(t \in [0, \infty)\).
\begin{proposition}\label{prop: fts qlc ordnen}
	Let \(\bullet \in \{pst, st\}\). If \(X^{ftd} \preceq_{\bullet} Y^{ftd}\) and \(X^{qlc} \preceq_{\bullet} Y^{qlc}\), then \(X \preceq_{\bullet} Y\).
	If \(|x - h(x)| \star \nu^X_t +|x - h(x)|\star \nu^Y_t < \infty\) for all \(t \in [0, \infty)\), then the statement also holds for \(\bullet \in \{cx, icx\}\).
\end{proposition}
\begin{proof}
	Let \(f \in \mathcal{F}_{pst}\) be bounded. Then, by the independence and Fubini's theorem,
	\begin{align}
	E[f(X)] 
	&= \iint f\left(\omega + \alpha\right) P(X^{ftd} \in \dd \omega) P(X^{qlc} \in \dd \alpha) \nonumber
	\\&\leq \iint f\left(\omega+ \alpha\right) P(Y^{ftd} \in \dd \omega) P(X^{qlc} \in \dd \alpha)\nonumber
	\\&= \iint f\left(\omega + \alpha\right) P(X^{qlc} \in \dd \alpha) P(Y^{ftd} \in \dd \omega)\label{eq: Fubini appl}
	\\&\leq \iint f\left(\omega + \alpha\right) P(Y^{qlc} \in \dd \alpha) P(Y^{ftd} \in \dd \omega)\nonumber
	= E[f(Y)].
	\end{align}
	Since the stochastic order \(\preceq_{pst}\) is generated by the class of bounded functions in \(\mathcal{F}_{pst}\), see \cite[pp. 81]{muller2002comparison}, we can conclude that \(X \preceq_{pst} Y\). The case \(\preceq_{st}\) follows identically. 
	
	For the convex cases, we note that each (increasing) convex function can be approximated pointwise in a monotone manner by (increasing) Lipschitz continuous convex functions.
	More precisely, for a function \(f \colon \mathbb{R}^d \to \mathbb{R}\) we set \(f_n (x) \triangleq \inf_{z \in \mathbb{R}^d} (f(z) + n |x - z|)\), which is the inf-convolution of \(f\). It is well-known that in the case where \(f\) is convex, the inf-convolution \(f_n\) is Lipschitz continuous and convex, \(f_n(x) \leq f_{n+1}(x) \leq f(x)\) and \(f_n \to f\) pointwise as \(n \to \infty\), see, e.g., \cite[Lemma 2]{10.2307/2161214}. 
	Moreover, if \(f\) is increasing, then \(f_n\) is also increasing. To see this, note that for \(x \leq y\) we have 
	\begin{equation}\label{eq: inf-conv increasing}
	\begin{split}
	f_n(x) = \inf_{z \in \mathbb{R}^d} (f(z) + n|x - z|) &\leq \inf_{z \in \mathbb{R}^d} (f(z + y - x) + n|y - (z + y - x)|) 
	\\&= \inf_{z \in \mathbb{R}^d} (f(z) + n|y - z|) = f_n(y). 
	\end{split}
	\end{equation}
	Thus, by the monotone convergence theorem, we may restrict ourselves to (increasing) Lipschitz continuous convex functions \(\mathbb{R}^{d \cdot n} \to \mathbb{R}\). Denote one of these by \(f\). 
	To use the same argumentation as in the case \(\preceq_{pst}\), we only have to verify the application of Fubini's theorem, see \eqref{eq: Fubini appl}.
	The assumptions \(|x - h(x)| \star \nu^X_t < \infty\) and \(|x - h(x)| \star \nu^Y_t < \infty\) imply that
	\(E[|Y^{ftd}_t|] < \infty\) and \(E[|X^{qlc}_t|] < \infty\). Hence, since all Lipschitz continuous functions are of linear growth, we have
	\begin{align*}
	\iint |f(x + y)|& P\left(\left(Y^{ftd}_{t_1}, ..., Y^{ftd}_{t_n}\right) \in \dd x\right)  P\left(\left(X^{qlc}_{t_1}, ..., X^{qlc}_{t_n}\right) \in \dd y\right) 
	\\&\leq \textup{const. } \left(1 + \sum_{k = 1}^n E\left[\left|Y^{ftd}_{t_k}\right|\right] + \sum_{j = 1}^n E\left[\left|X^{qlc}_{t_j}\right|\right] \right) < \infty
	\end{align*}
	 for all \(0 \leq t_1 < ... < t_n < \infty\). Therefore, we can apply Fubini's theorem and the claim follows similar to the case \(\preceq_{pst}\).
\end{proof}
\subsection{Stochastic Orders for the Fixed Times of Discontinuity}
For the fixed times of discontinuity it suffices to order each summand separately. Let \(X\) and \(Y\) be as in the previous section. 
\begin{proposition}
	If for all \(t \in [0, \infty)\) it holds that \(\Delta X^{ftd}_t \preceq_{st} \Delta Y^{ftd}_t\), then \(X^{ftd} \preceq_{pst} Y^{ftd}\). Moreover, if for all \(t \in [0, \infty)\) it holds that \(|x| \1_{J^Y} \star \nu^Y_t < \infty\) and 
	 \(\Delta X^{ftd}_t \preceq_{(i)cx} \Delta Y^{ftd}_t\), then \(X^{ftd} \preceq_{(i)cx} Y^{ftd}\).
\end{proposition}
Here, \(\Delta X^{ftd}_t \preceq_{\bullet} \Delta Y^{ftd}_t\) refers to stochastic orders of \(\mathbb{R}^d\)-valued random variables. 
\\\\
\begin{proof}
	In the case \(\preceq_{pst}\) the claim follows from Strassen's theorem \cite[Theorem 1]{kamae1977}: 
	We find a probability space which supports two sequences \((\Delta X_t)_{t \in J^X \cup J^Y}\) and \((\Delta Y_t)_{t \in J^X \cup J^Y}\) of independent random variables such that \(\Delta X_t\) has law \eqref{law jump}, \(\Delta Y_t\) has law \eqref{law jump} with \(\nu^X\) replaced by \(\nu^Y\) and a.s. \(\Delta X_t \leq \Delta Y_t\) for all \(t \in J^X \cup J^Y\).
	Set \(J_t \triangleq (J^X \cup J^Y) \cap [0, t]\) for \(t \in [0, \infty)\).
	We claim that the sums \(\sum_{s \in J_t} \Delta X_s\) and \(\sum_{s \in J_t} \Delta Y_s\) converge a.s. To see this, set \(Z_s \triangleq \Delta X_s \1_{\{|\Delta X_s| \leq 1\}}\) and note, by \cite[II.5.5.(i)]{JS}, that we have
	\[
	\sum_{s \in J_t} P(Z_s \not = \Delta X_s) = \sum_{s \in J_t} P(|\Delta X_s| > 1) \leq \nu^X([0, t] \times \{|x| > 1\}) < \infty.
	\]
	Hence, by the Borel-Cantelli lemma, \(\sum_{s \in J_t} \Delta X_s\) converges a.s. if, and only if, \(\sum_{s \in J_t} Z_s\) converges a.s.
	Since \(h\) is a truncation function there exists an \(\epsilon > 0\) such that \(h(x) = x\) on \(\{|x| \leq \epsilon\}\).
	Since
	\begin{align*}
\sum_{s \in J_t} E[|Z_s|] 
&\leq |x - h(x)| \1_{\{|x| \leq 1\}} \star \nu^X_t + |h(x)| \star \nu^X_t 
	\\&\leq \textup{const. } \nu^X([0, t] \times \{\epsilon < |x| \leq 1\}) + |h(x)| \star \nu^X_t < \infty,
	\end{align*}
	due to \cite[II.5.5.(i)]{JS} and the assumption that \(|h(x)| \star \nu^X_t < \infty\), we conclude that \(\sum_{s \in J_t} Z_s\) converges a.s.
	Thus, \(\sum_{s \in J_t} \Delta X_s\) converges a.s. and \(\sum_{s \in J_t} \Delta Y_s\) converges a.s. by the same arguments.
	We also claim that \(\sum_{s \in J_\cdot} \Delta X_s\) has the same law as \(X^{ftd}\) and \(\sum_{s \in J_\cdot} \Delta Y_s\) has the same law as \(Y^{ftd}\).
	By the \cadlag paths, we only have to show that the processes have the same finite dimensional distributions, see, for instance, \cite[Lemma VI.3.19]{JS}.
	 Now, take \(0 = t_0 \leq t_1 < ... < t_n < \infty\) and let \(A\in \mathbb{R}^n \otimes \mathbb{R}^n\) be the lower triangular matrix with \(A_{ij} = 1\) for all \(i \geq j\).
	Then, \begin{align*}
	\left(X_{t_1}^{ftd}, ..., X^{ftd}_{t_n}\right)^\text{tr} &= A \left(X^{ftd}_{t_1} - X^{ftd}_{t_0}, ..., X^{ftd}_{t_n} - X^{ftd}_{t_{n-1}}\right)^\textup{tr},\\
	\left(\sum_{s \in J_{t_1}} \Delta X_s, ..., \sum_{s \in J_{t_n}} \Delta X_s\right)^\textup{tr} &=  A\left(\sum_{s \in J_{t_1}\backslash J_{t_0}} \Delta X_s, ..., \sum_{s \in J_{t_n} \backslash J_{t_{n-1}}} \Delta X_s\right)^\textup{tr}.
	\end{align*}
	Now, using the uniqueness theorem for characteristic functions and the fact that the entries of the right hand vectors are independent, it suffices to show that for all \(0 \leq s < t < \infty\) the sum \(\sum_{r \in J_t \backslash J_s} \Delta X_r\) has the same law as \(X^{ftd}_t - X^{ftd}_s\).
	This, however, follows from the fact that for all \(u \in \mathbb{R}^d\) 
	\begin{align*}
	E&\left[ \exp \left( \sqrt{-1}\ \bigg\langle u, \sum_{r \in J_t\backslash J_s} \Delta X_r\bigg\rangle\right)\right] 
	\\&\hspace{1.6cm}= \prod_{r \in (s, t]} \left(1 + \int \left(\exp\left(\sqrt{-1} \langle u, x\rangle\right) - 1 \right) \nu^X(\{r\}\times \dd x)\right),
	\end{align*}
	\cite[Theorem II.5.2]{JS} and the uniqueness theorem for characteristic functions.
	Thus, we conclude that \(\sum_{s \in J_\cdot} \Delta X_s\) has the same law as \(X^{ftd}\). The same argument also shows that \(\sum_{s \in J_\cdot} \Delta Y_s\) has the same law as \(Y^{ftd}\).
	Now, a.s. \(\sum_{s \in J_\cdot} \Delta X_s \leq \sum_{s \in J_\cdot} \Delta Y_s\) implies the stochastic order \(X^{ftd} \preceq_{pst} Y^{ftd}\).

	Let us presume that \(|x| \1_{J^Y} \star \nu^Y_t < \infty\) for all \(t \in [0, \infty)\).
	If \(\Delta X_t \preceq_{cx} \Delta Y_t\) for all \(t \in J^X \cup J^Y\), then, by Strassen's theorem \cite[Theorem 3.4.2]{muller2002comparison}, we find a probability space which supports two sequences 
	\((\Delta X_t)_{t \in J^X \cup J^Y}\) and \((\Delta Y_t)_{t \in J^X \cup J^Y}\) of independent random variables such that \(\Delta X_t\) has law \eqref{law jump} and \(\Delta Y_t\) has law \eqref{law jump} with \(\nu^X\) replaced by \(\nu^Y\) and a.s. \(E[\Delta Y_t |\mathscr{H}] = X_t\) for all \(t \in J^X \cup J^Y\), where \(\mathscr{H} \equiv \sigma(\Delta X_s, s \in J^X \cup J^Y)\).
	Fix \(t \in [0, \infty)\) and set \(J_t\) as above.
	The assumption \(|x| \1_{J^Y} \star \nu^Y_t < \infty\) implies that 
	\[E \left[\sum_{s \in J_t} | \Delta Y_s|\right] < \infty.\]
	Thus, we have a.s.
	\begin{align*}
	E\left[ \sum_{s \in J_t} \Delta Y_s \bigg| \mathscr{H} \right] &= \sum_{s \in J_t} E\left[\Delta Y_s | \mathscr{H}\right] = \sum_{s \in J_t} \Delta X_s.
	\end{align*}
	Since \(\sum_{s \in J_\cdot} \Delta X_s\) has the same law as \(X^{ftd}\) and \(\sum_{s \in J_\cdot} \Delta Y_s\) has the same law as \(Y^{ftd}\), 
	we conclude \(X^{ftd} \preceq_{cx} Y^{ftd}\) from the conditional Jensen's inequality.
	The case \(\preceq_{icx}\) follows similarly.
\end{proof}
Explicit conditions for stochastic orders of \(\mathbb{R}^d\)-valued random variables can be found in \cite{muller2002comparison}.
Next, we will discuss the quasi-left continuous parts.

\subsection{Stochastic Orders for Quasi-Left Continuous PIIs}\label{Main Results}
Let \(X\) and \(Y\) be quasi-left continuous PIIs with characteristics \((B^X, C^X, \nu^X)\) and \((B^Y, C^Y, \nu^Y)\) respectively.
First, we assume that the discontinuous parts of \(X\) and \(Y\) are of finite variation and that \(X\) and \(Y\) have first moments, i.e. for all \(t \in [0, \infty)\) we assume that \(|x| \star \nu^X_t +|x| \star \nu^Y_t < \infty\).
In this case, the PIIs \(X\) and \(Y\) have a decomposition
\begin{align*}
X &= B^X - h(x) \star \nu^X + X^c + x \star \mu^X,\\
Y &= B^Y + h(x) \star \nu^Y + Y^c + x \star \mu^Y,
\end{align*}
where \(X^c\) is a Wiener process with variance function \(C^X\) and \(Y^c\) is a Wiener process with variance function \(C^Y\) in the sense of \cite[Definition I.4.9]{JS}.
In particular, \(X^c\) is independent of \(x \star \mu^X\) and \(Y^c\) is independent of \(x \star \mu^Y\), see \cite[Lemma 13.6]{Kallenberg}.
\begin{proposition}\label{prop: Gauss CPP}
	Let \(\bullet \in \{pst, st, cx, icx\}\). If \(B^X - h(x) \star \nu^X + X^c \preceq_{\bullet} B^Y - h(x) \star \nu^Y + Y^c\) and \(x \star \mu^X \preceq_{\bullet} x \star \mu^Y\), then \(X \preceq_\bullet Y\).
\end{proposition}
\begin{proof}
	This follows as in the proof of Proposition \ref{prop: fts qlc ordnen}.
\end{proof}
For all \(t \in [0, \infty)\) the random variable \(B^X_t - h(x) \star \nu^X_t + X^c_t\) is Gaussian with expectation \(B^X_t - h(x)\star \nu^X_t\) and covariance matrix \(C^X_t\), and the random variable \(B^Y - h(x) \star \nu^Y_t + Y^c_t\) is Gaussian with expectation \(B^Y_t - h(x) \star \nu^Y_t\) and covariance matrix \(C^Y_t\). Hence, the question when \(B^X - h(x) \star \nu^X + X^c \preceq_{\bullet} B^Y - h(x) \star \nu^X + Y^c\) is a question when two Gaussian vectors are stochastically ordered. This question, however, is well-studied, see, e.g., \cite{muller2002comparison}, and we restrict ourselves to the Poisson sums by assuming that
\[
B^X - h(x) \star \nu^X =  B^Y - h(x) \star \nu^Y = 0,\quad X^c = Y^c = 0.
\]

We note the following technical observation:
\begin{lemma} There exist a decomposition
	\begin{align*}
	K^\bullet(t, \dd x) \dd A_t = \nu^{\bullet} (\dd t \times \dd x)
	\end{align*}
	where \(K^{\bullet}\) is a Borel transition kernel from \([0, \infty)\) to \(\mathbb{R}\) and \(A\) is an increasing continuous function of finite variation. 
\end{lemma}
\begin{proof}
	It is well-known that such a decomposition exists, see \cite[II.1.2, Theorem II.1.8]{JS}. That we can take the same \(A\) for both decompositions is a consequence of the Radon-Nikodym theorem.
\end{proof}
We write \(K^X \preceq_{\bullet} K^Y\) if \(\int f(x) K^X(t, \dd x) \leq \int f(x) K^Y(t, \dd x)\) for \(\dd A_t\)-a.a. \(t \in [0, \infty)\) and all Lipschitz continuous \(f \in \mathcal{F}^1_{\bullet}\) with \(|f(x)| \leq \textup{const. } |x|\). We stress that for Lipschitz continuous functions \(f\) the growth condition \(|f(x)| \leq \textup{const. } |x|\) is equivalent to \(f(0) = 0\).
\begin{remark}
	If \(\dd A_t\)-a.e. \(K^X(\cdot, \mathbb{R}^d) < \infty\) and \(K^X(\cdot, \mathbb{R}^d) = K^X(\cdot, \mathbb{R}^d)\), then \(K^X\) and \(K^Y\) are ordered if, and only if, for \(\dd A_t\)-a.a. \(t \in [0, \infty)\) the random variables with laws \(K^X(t, \mathbb{R}^d)^{-1} K^X(t, \dd x)\) and  \(K^Y(t, \mathbb{R}^d)^{-1} K^Y(t, \dd x)\) are ordered.
	Thus, in the one-dimensional case \(d = 1\), we have the following characterizations:
	\begin{align*}
	K^X \preceq_{st} K^Y \quad &\Longleftrightarrow \quad K^Y(t, (- \infty, x]) \leq K^X(t, (- \infty, x]) \\&\hspace{2.5cm}\text{ for all } x \in \mathbb{R} \text{ and } \dd A_t\text{-a.a. } t \in [0, \infty),\\
	K^X \preceq_{icx} K^Y \quad &\Longleftrightarrow \quad \int (y - x)_+ K^X(t, \dd y) \leq \int (y - x)_+K^Y(t, \dd y) \\&\hspace{2.5cm}\text{ for all } x \in \mathbb{R} \text{ and } \dd A_t\text{-a.a. } t \in [0, \infty),\\
	K^X \preceq_{cx} K^Y \quad &\Longleftrightarrow \quad \int (y - x)_+ K^X(t, \dd y) \leq \int (y - x)_+K^Y(t, \dd y) \\&\hspace{1.07cm}\text{ and } \int y K^X(t, \dd y) = \int y K^Y(t, \dd y) \\&\hspace{2.5cm}\text{ for all } x \in \mathbb{R} \text{ and } \dd A_t\text{-a.a. } t \in [0, \infty).
	\end{align*}
	These characterizations can be deduced from \cite[Theorems 1.2.8, 1.5.3 and 1.5.7]{muller2002comparison} together with the fact that the stochastic order \(\preceq_{st}\) is generated by the increasing functions of class \(C^2_b\), see \cite[Theorems 2.4.2, 2.5.5 and 3.3.10]{muller2002comparison}.
In Lemma \ref{lem: eq cond} below we will see a generalization of the first equivalence to cases where \(K^X(\cdot, \dd x)\) and \(K^Y(\cdot, \dd x)\) are not finite nor have the same mass.
\end{remark}

%

\begin{theorem}\label{theo: Main Uni}
	Let \(\bullet \in \{st, cx, icx\}\). It holds that
	\[K^X \preceq_{\bullet} K^Y\quad \Longrightarrow \quad X \preceq_{\bullet} Y.\]
	If \(X\) and \(Y\) are L\'evy processes, then also
	\[K^X \preceq_{\bullet} K^Y\quad \Longleftarrow \quad X \preceq_\bullet Y,\]
	where \(K^X\) and \(K^Y\) are the L\'evy measures.
\end{theorem}
\begin{proof}
	Let us start with the first claim.
The following is a version of \cite[Lemma 31]{bergenthum2007b}. 
\begin{lemma}\label{lem: pointwise}
If for all \(0 \leq s < t< \infty\) it holds that
	\begin{align}\label{eq: to show main1}
	X_t - X_s \preceq_\bullet Y_t - Y_s,
	\end{align}
	then \(X \preceq_{\bullet} Y\).
\end{lemma}
\begin{proof}
	We use an induction argument. Take \(0 \leq s < t < \infty\).
	Note that \((X_s, X_t) = (X_s, X_s) + (0, X_t - X_s)\) and \((Y_s, Y_t) = (Y_s, Y_s) + (0, Y_t - Y_s)\), which are sums of independent random variables due the independent increment property of \(X\) and \(Y\).
	Since the monotone and convex stochastic orders are closed w.r.t. identical concentration, see \cite[Theorems 3.3.10 and 3.4.4]{muller2002comparison}, \(X_s \preceq_{\bullet} Y_s\) implies \((X_s, X_s) \preceq_{\bullet} (Y_s, Y_s)\). 
Moreover, since all stochastic orders under consideration are closed w.r.t. independent concentration, see \cite[Theorems 3.3.10 and 3.4.4]{muller2002comparison}, \(X_{t}- X_s \preceq_{\bullet} Y_t - Y_s\) implies \((0, X_t - X_s) \preceq_{\bullet} (0, Y_t- Y_s)\).
Thus, by the convolution property, see again \cite[Theorems 3.3.10 and 3.4.4]{muller2002comparison}, \((X_s, X_t) \preceq_\bullet (Y_s, Y_t)\) follows.

Now, take \(0 \leq t_1 < ..., < t_n < \infty\).
We have 
\begin{align*}
(X_{t_1}, ..., X_{t_n}) &= (X_{t_1}, ..., X_{t_{n-1}}, X_{t_{n-1}}) + (0, ..., 0, X_{t_n} - X_{t_{n-1}}),\\
(Y_{t_1}, ..., Y_{t_n}) &= (Y_{t_1}, ..., Y_{t_{n-1}}, Y_{t_{n-1}}) + (0, ..., 0, Y_{t_n} - Y_{t_{n-1}}).
\end{align*}
By the independent increment property, the vectors on the right hand sides are independent.
Using the induction hypothesis and the same arguments as above concludes the proof.
\end{proof} 
	Thus, it suffices to show \eqref{eq: to show main1}.
	Our main tool is an interpolation formula in the spirit of \cite{Houdre1998}. A related formula was used in \cite{Bauerle2007} to prove a supermodular stochastic order for L\'evy processes.
	
	Let the process \(Z(\alpha)\) be a PII with characteristics \((\alpha h(x) \star \nu^{Y} + (1 - \alpha) h(x) \star \nu^{X}, 0, \alpha \nu^{Y} + (1- \alpha) \nu^{X})\), see \cite[Theorem II.5.2]{JS} for the existence, and
	set 
	\begin{align*}
	\mathcal{L}^\bullet_{s, t} f (x) \triangleq \int_s^t \int \left( f(x + y) - f(x)\right)K^{\bullet}(r,\dd y)\dd A_r
	\end{align*}
	for Lipschitz continuous functions \(f \colon \mathbb{R}^d \to \mathbb{R}\). The assumption \(|x| \star \nu^\bullet_t < \infty\) implies that \(\mathcal{L}^\bullet_{s,t} f\) is well-defined.
	\begin{lemma}\label{interpolation formula}
		For all Lipschitz continuous \(f \colon \mathbb{R}^d \to \mathbb{R}\) it holds that
		\begin{align*}
		E&\left[f\left(Y_t - Y_s\right)\right] - E\left[f\left(X_t - X_s\right)\right] 
		\\&\hspace{1cm}= \int_0^1 \int \left(\mathcal{L}^Y_{s, t} f (z) - \mathcal{L}^X_{s, t} f(z)\right) P(Z_t (\alpha) - Z_s(\alpha) \in \dd z) \dd \alpha.
		\end{align*}
	\end{lemma}
	\begin{proof}
		For \(f \in C^2_b(\mathbb{R}^d, \mathbb{R})\) the claim follows from \cite[Proposition 1]{Houdre1998}.
		The claim for Lipschitz continuous \(f\) follows by approximation: 
		First, we approximate \(f\) by bounded Lipschitz continuous functions. Set \(f_n \triangleq n\) on \(\{x \in \mathbb{R}^d \colon f(x) \geq n\}\), \(f_n \triangleq- n\) on \(\{x \in \mathbb{R}^d \colon f(x) \leq -n\}\) and \(f_n \triangleq f\) otherwise. It is routine to check that \(f_n\) is Lipschitz continuous with the same Lipschitz constant as \(f\). 
		Thus, for \(U \in \{X, Y\}\) we have \(|f_n(U_t - U_s)| \leq \textup{const. } (1 + |U_t - U_s|)\) and \(|f_n(z + y) - f_n(z)| \leq \textup{const. } |y|\), where both constants are independent of \(n\).
		Since, using our assumptions, \(|U_t - U_s|\) is integrable w.r.t. \(P\) and \(|y|\) is integrable w.r.t. \(\1_{[0, 1]}(\alpha) \1_{(s, t]}(r) K^U(r, \dd y) \dd A_r P(Z_t(\alpha) - Z_s(\alpha) \in \dd z) \dd \alpha\), 
		we can apply the dominated convergence theorem to obtain
		\begin{align}\label{approx1}
		\lim_{n \to \infty} E[f_n(U_t - U_s)] = E[f(U_t - U_s)]
		\end{align}
		and
		\begin{equation}\label{approx2}
		\begin{split}
		\lim_{n \to \infty} \int_0^1 &\int \mathcal{L}^U_{s, t} f_n (z)P(Z_t (\alpha) - Z_s(\alpha) \in \dd z) \dd \alpha 
		\\&= \int_0^1 \int \mathcal{L}^U_{s, t} f (z)P(Z_t (\alpha) - Z_s(\alpha) \in \dd z) \dd \alpha.
		\end{split}
		\end{equation}
		Hence, the claim holds for all bounded Lipschitz continuous function. We approximate a second time. Let \(\phi\) be the standard mollifier, i.e. \(\phi(x) = c \exp(  (|x|^2 - 1)^{-1}) \1_{\{|x| \leq 1\}}\), where \(c\) is a normalization constant, and set \(\phi_n (x) \triangleq n^{d} \phi(n x)\).
		Define \(f_n \triangleq f * \phi_n\), where \(*\) denotes the convolution.
		It is well-known that \(f_n \in C^\infty(\mathbb{R}^d, \mathbb{R})\), see, for instance, \cite[Theorem II.6.30]{rudin2006functional}. In particular, since \(f\) is bounded and \(\frac{\dd}{\dd x_i}\phi_n\) has compact support, we deduce from the formula \(\frac{\dd }{\dd x_i} (f * \phi_n) = f * \frac{\dd}{\dd x_i}\phi_n\), see, for instance, \cite[Theorem II.6.30]{rudin2006functional}, that \(f * \phi_n \in C^2_b(\mathbb{R}^d, \mathbb{R})\).
		It is routine to check that \(f * \phi_n\) is Lipschitz continuous with the same Lipschitz constant as \(f\). 
		Moreover, for all \(x \in \mathbb{R}^d\), we have 
		\begin{align*}
		\left| (f * \phi_n) (x) - f(x)\right| &\leq \int_{|z| \leq 1} \phi(z) \left| f(x) - f\left(x - \frac{z}{n}\right)\right| \dd z
		\\&\leq \textup{const. } \int_{|z| \leq 1}  \frac{\phi (z)|z|}{n} \dd z
		\leq \frac{1}{n} \xrightarrow{\quad n \to \infty\quad} 0,
		\end{align*}
		i.e. \(f_n \to f\) pointwise as \(n \to \infty\).
		Thus, as above, we conclude from the dominated convergence theorem that \eqref{approx1} and \eqref{approx2} hold. This finishes the proof. 
	\end{proof}
	
	For all \(f \in \mathcal{F}^1_\bullet\) and \(x \in \mathbb{R}^d\) we have \(f(x + \cdot) - f(x) \in \mathcal{F}^1_\bullet\).
	Thus, \(K^X \preceq_{\bullet} K^Y\) implies \(\mathcal{L}^Y_{s, t} f(z) - \mathcal{L}^Y_{s, t} f(z) \geq 0\) for all \(0 \leq s < t < \infty\), \(z \in \mathbb{R}^d\) and all Lipschitz continuous \(f\in \mathcal{F}^1_\bullet\). In particular, this yields \(E[f(X_t - X_s)] \leq E[f(Y_t - Y_s)]\) by Lemma \ref{interpolation formula}.
	
	We note that the stochastic order \(\preceq_{st}\) is generated by the increasing functions in \(C^2_b(\mathbb{R}^d, \mathbb{R})\).
	To see this note first that the stochastic order \(\preceq_{st}\) is generated by all increasing functions in \(C_b(\mathbb{R}^d, \mathbb{R})\), see \cite[Theorems 2.4.2 and 3.3.10]{muller2002comparison}. Then, applying a mollification argument to each of these functions yields the claim, see \cite[Theorem 2.5.5]{muller2002comparison} or the proof of Lemma \ref{interpolation formula}.
	Since all \(f \in C^2_b(\mathbb{R}^d, \mathbb{R})\) are Lipschitz continuous, we conclude \(X_t - X_s \preceq_{st} Y_t - Y_s\).
	
	To obtain the claim for \(\preceq_{(i)cx}\) we can use the fact that all \(f \in \mathcal{F}^1_{(i)cx}\) can be approximated by (increasing) convex Lipschitz continuous functions in a monotone manner, see the proof of Proposition \ref{prop: fts qlc ordnen}.	
	Hence, \(X_t - X_s \preceq_{(i)cx} Y_t - Y_s\) follows from the monotone convergence theorem.
	
	For the rest of the proof we assume that \(X\) and \(Y\) are L\'evy processes with L\'evy measures \(K^X\), \(K^Y\) respectively.
	\begin{lemma}
		For \(U \in \{X, Y\}\) and all Lipschitz continuous \(f\colon \mathbb{R}^d \to \mathbb{R}\) the process
		\begin{align*}
		f&(U_\cdot) - f(0) - \int_0^\cdot \int  \left(f(U_s + x) - f(U_s)\right) K^U(\dd x) \dd s
		\end{align*}
		is a martingale.
	\end{lemma}
	\begin{proof}
		The same approximation arguments as used in the proof of Lemma \ref{interpolation formula} yield that it suffices to show the claim for \(f \in C^2_b(\mathbb{R}^d, \mathbb{R})\).
		In this case, It\^o's formula yields that the process is a local martingale.
		Using that \(|f(x)| \leq \textup{const.} (1 + |x|)\) and \(|f(y + x) - f(y)| \leq \textup{const.} |x|\) yields
		\begin{align*}
		\sup_{s \in [0, t]}& \left| f(U_s) - f(0) - \int_0^s \int  \left(f(U_r + x) - f(U_r)\right) K^U(\dd x) \dd r\right|
		\\&\hspace{4.8cm}\leq\text{ const.} \left(1 + t + \sup_{s \in [0, t]} |U_s| \right),
		\end{align*}
		where the constant is uniform.
		Since \(U\) is a L\'evy process, \(\sup_{s \in [0, t]} |U_s|\) is integrable if \(\int |x - h(x)| K^U(\dd x) < \infty\), see \cite[Corollary 25.8, Theorem 25.18]{Sato99}. Hence, the martingale property follows from the dominated convergence theorem.
	\end{proof}
	Thanks to this observation and Fubini's theorem, for all Lipschitz continuous functions \(f\) with \(f(0) = 0\) it holds that
	\begin{align*}
	\lim_{t \downarrow 0}\frac{E[f(X_t)]}{t} &= \lim_{t \downarrow 0} \frac{1}{t} \int_0^t E \left[ \int \left(f (X_s + x) - f(X_s) \right)K^X(\dd x)\right]  \dd s 
	\\&= E\left[ \int \left(f(X_0 + x) - f(X_0)\right) K^X(\dd x)\right]
	= \int f(x) K^X(\dd x).
	\end{align*}
	Using the same arguments for \(X\) replaced by \(Y\), we obtain for all Lipschitz continuous \(f\) with \(f(0) = 0\) that
	\begin{align*}
	0 \leq \lim_{t \downarrow 0} \frac{E[f(Y_{t})] - E[f(X_{t})]}{t} &= \int f(x) K^Y(\dd x) - \int f(x) K^X(\dd x).
	\end{align*}
	This concludes the proof of Theorem \ref{theo: Main Uni}.
\end{proof}

For compound Poisson processes with equal jump intensity a related result was shown in \cite[Lemma 3.2]{bergenthum2007} with a different proof. By 
a modification of the L\'evy measure and approximation arguments, the conditions are generalized more general L\'evy processes, see \cite{bergenthum2007} for details.
In fact, Theorem \ref{theo: Main Uni} shows that the claims of \cite[Lemma 3.2]{bergenthum2007} hold for all finite variation pure-jump L\'evy processes without modifying the L\'evy measures.

Now, we will relax the integrability assumptions. To be precise, let \(X\) and \(Y\) be quasi-left continuous PIIs with characteristics \((B^X, C^X, \nu^X)\) and \((B^Y, C^Y, \nu^Y)\) such that \(|x - h(x)| \star \nu^\bullet_t < \infty\) for all \(t \in [0, \infty)\) and suppose that \(h\) is continuous.

Take a sequence \((G_n)_{n \in \mathbb{N}} \subseteq \mathbb{R}^d\) of Borel sets such that \(\1_{G_n}\) is vanishing in a neighborhood of the origin and \(\bigcup_{n \in \mathbb{N}} G_n \supseteq \mathbb{R}^d\backslash \{0\}\).
\begin{theorem}\label{theo: Approx 1}
	\begin{enumerate}
		\item[\textup{(i)}]
		Suppose that for all \(n \in \mathbb{N}\) and \(t \in [0, \infty)\) it holds that \(\1_{G_n}\cdot K^X \preceq_{st} \1_{G_n} \cdot K^Y, C^X = C^Y\) and
		\begin{align}\label{drift cond trunc}
		h(x) \1_{G_n}(x) \star \nu^{Y}_t - h(x)\1_{G_n}(x) \star \nu^{X}_t \leq B^{Y}_t - B^{X}_t.
		\end{align}
		Then \(X \preceq_{st} Y\).
		\item[\textup{(ii)}] Suppose that for all \(n \in \mathbb{N}\) and \(t \in [0, \infty)\) it holds that \(\1_{G_n}\cdot K^X \preceq_{icx} \1_{G_n}\cdot K^Y, C^Y_t - C^X_t\) is non-negative definite and \eqref{drift cond trunc} is satisfied.
		Then \(X \preceq_{icx} Y\). If, additionally, for all \(t \in [0, \infty)\)
		\begin{align}\label{eq: equal moment}
		B^X_t + |x - h(x)| \star \nu^X_t = B^Y_t + |x - h(x)| \star\nu^Y_t,
		\end{align}
		then \(X \preceq_{cx} Y\).
	\end{enumerate}
\end{theorem}
\begin{proof}
	We start with some general observations.
	Define the truncated processes
	\begin{align}\label{Z trunc}
	Z^{\bullet}(n) \triangleq B^{\bullet} &+ h \1_{G_n} \star (\mu^\bullet - \nu^\bullet) 
	+ (x - h(x))  \1_{G_n} \star \mu^\bullet.
	\end{align}
	\begin{lemma}\label{lem: qlc}
		\(Z^{\bullet}(n)\) is a PII with characteristics \((B^{\bullet}, 0, \1_{G_n}\cdot \nu^{\bullet})\).
		Moreover, if \(Z^c\) is a Wiener process with covariance function \(C\) which is independent of \(Z^\bullet(n)\), then \(Z^\bullet (n) + Z^c\) is a PII with characteristics \((B^\bullet, C, \1_{G_n} \cdot \nu^\bullet)\).
	\end{lemma}
	\begin{proof}
		The first claim follows from \cite[Theorem II.5.10]{JS} and the second claim follows from \cite[Lemma II.2.44, Theorem II.5.2]{JS}.
	\end{proof}
	We pose ourselves in the setting of (i).
	Let \(Z^c\) be as in the previous lemma with \(C \triangleq C^X = C^Y\) and set \(\widehat{Z}^\bullet(n) \triangleq Z^{\bullet}(n) + Z^c\).
	By Lemma \ref{lem: pointwise}, it suffices to show that \(X_t - X_s \preceq_{st} Y_t - Y_s\) for all \(0 \leq s < t < \infty\).
	The same arguments as used in the proof of Theorem \ref{theo: Main Uni} together with Proposition \ref{prop: Gauss CPP} yields that \(\widehat{Z}^X(n)_t - \widehat{Z}^X(n)_s \preceq_{st} \widehat{Z}^Y(n)_t - \widehat{Z}^Y(n)_s\) for all \(n \in \mathbb{N}\). 
	It follows from \cite[Theorem VII.3.4]{JS} that \(\widehat{Z}^X(n)\) convergences in law to \(X\) and \(\widehat{Z}^Y (n)\) converges in law to \(Y\) as \(n \to \infty\).
	Since the stochastic order \(\preceq_{st}\) for \(\mathbb{R}^d\)-valued random variables is closed under weak convergence, see \cite[Theorem 3.3.10]{muller2002comparison}, we conclude that \(X_t - X_s \preceq_{st} Y_t - Y_s\). This proves \(X \preceq_{st} Y\) by Lemma \ref{lem: pointwise}.
	
	Next, we prove (ii). First, we do not assume \eqref{eq: equal moment}. Let \(Z^{\bullet, c}\) be a Wiener process with covariance function \(C^\bullet\) independent of \(Z^\bullet(n)\) for all \(n \in \mathbb{N}\) and set \(\widetilde{Z}^\bullet(n) \triangleq Z^\bullet(n) + Z^{\bullet, c}\). 
	The same arguments as used in the proof of Theorem \ref{theo: Main Uni} together with Proposition \ref{prop: Gauss CPP} yields that \(\widetilde{Z}^X(n)_t - \widetilde{Z}^X(n)_s \preceq_{icx} \widetilde{Z}^Y(n)_t - \widetilde{Z}^Y(n)_s\) for all \(n \in \mathbb{N}\).
	We note that \cite[Theorem VII.3.4]{JS} implies that \(\widetilde{Z}^X(n)\) convergences in law to \(X\) and \(\widetilde{Z}^Y (n)\) converges in law to \(Y\) as \(n \to \infty\). For \(U \in \{X, Y\}\) the dominated convergence theorem yields that
	\begin{align*}
E \left[ Z^U (n)_t - Z^U (n)_s \right] &= B^U_t - B^U_s + \int (x - h(x)) \1_{G_n} \nu^U((s, t] \times \dd x)
	\\&\xrightarrow{n \to \infty} B^U_t - B^U_s + \int (x - h(x)) \nu^U ((s, t] \times \dd x) 
	\\&= E[U_t - U_s].
	\end{align*}
	Hence, we conclude from \cite[Theorem 3.4.6]{muller2002comparison} that \(X_t - X_s \preceq_{icx} Y_t - Y_s\) and, therefore, \(X \preceq_{icx} Y\) by Lemma \ref{lem: pointwise}.
	
	Finally, suppose that \eqref{eq: equal moment} holds.
	In this case, \(E[X_t - X_s] = E[Y_t - Y_s]\) and we deduce from \cite[Theorem 3.4.2]{muller2002comparison} that \(X_t - X_s \preceq_{cx} Y_t - Y_s\). Again due to Lemma \ref{lem: pointwise}, this yields \(X \preceq_{cx} Y\). 
\end{proof}
For the stochastic order \(\preceq_{cx}\) it might be natural to assume that \(\1_{G_n}\cdot K^X \preceq_{cx} \1_{G_n} \cdot K^Y\). However, in this case \(\int_{G_n} x K^X(t, \dd x) = \int_{G_n} x K^Y(t, \dd x)\) has to hold. This seems to be a restrictive assumption on the sequence \((G_n)_{n \in \mathbb{N}}\). On the contrary, \eqref{eq: equal moment} is a necessary condition for \(X \preceq_{cx} Y\) such that we consider our conditions for \(\preceq_{cx}\) to be weaker.
\section{Explicit Conditions for Quasi-Left Continuous PIIs}\label{sec: M}
In this section we give explicit conditions and we present alternative proofs via coupling arguments.
\subsection{A Majorization Condition for the Monotone Stochastic Order}
We suppose that \(d = 1\). Let \(X\) and \(Y\) be quasi-left continuous PIIs with characteristics \((B^X, C, \nu^X)\) and \((B^Y, C, \nu^Y)\).

Our approach is based on a coupling constructed via the It\^o map which relates L\'evy measures to a reference L\'evy measure. 
The main result of this section is the following
\begin{theorem}\label{theo: Mnew}
	Assume that
	\begin{equation}\label{main coupl st}
	\begin{split}
	K^Y (t, (- \infty, x]) &\leq K^X(t, (- \infty, x]),\quad x < 0,\\
	K^X(t, [x, \infty)) &\leq K^Y(t, [x, \infty)),\quad x > 0,
	\end{split}
	\end{equation}
	for \(\dd A_t\)-a.a. \(t \in [0, \infty)\).
	Moreover, assume that for all \(t \in [0, \infty)\)
	\begin{align}
	|h(&x)| \star \nu^{X}_t + |h(x)| \star \nu^{Y}_t < \infty,\label{fv}\\
	h(x) &\star \nu^{Y}_t - h(x) \star \nu^{X}_t \leq B^{Y}_t - B^{X}_t.\label{drift cond new}
	\end{align}
	Then, \(X \preceq_{pst} Y\).
\end{theorem}
The stochastic order \(\preceq_{pst}\) is stronger than the stochastic order \(\preceq_{st}\). In this regard, Theorem \ref{theo: Mnew} brings a new condition. 
\\\\
\begin{proof}
	We show that there exists a probability space which supports copies of \(X\) and \(Y\) such that a.s. \(X_t \leq Y_t\) for all \(t \in [0, \infty)\). 
	This clearly implies \(X \preceq_{pst}Y\).
	
	
	Let \(F\) be the L\'evy measure of a 1-stable L\'evy process, i.e. \(F(\dd x) = \frac{1}{|x|^{2}}\dd x\),
	and set 
	\begin{align}\label{ito map}
	\rho^\bullet (t, x) \triangleq \begin{cases}
	\sup \left( y \in [0, \infty)\colon K^\bullet (t, [y, \infty)) \geq \frac{1}{|x|}\right),&x > 0,\\
	0,&x = 0,\\
	- \sup \left( y \in [0, \infty)\colon K^\bullet (t, (- \infty, -y]) \geq \frac{1}{|x|}\right),&x < 0.
	\end{cases}
	\end{align}
	Adapting the terminology in \cite{stroock2010}, the function \(\rho^\bullet\) is called It\^o map.
	Now,
	\begin{align}\label{eq: ito}
	K^\bullet (t, G) =  \int \1_G (\rho^\bullet(t, x)) F(\dd x),\quad G \in \mathscr{B}(\mathbb{R}),
	\end{align}
	see \cite[Theorem 9.2.4]{stroock2010}. Set \(\nu^L(\dd t \times \dd x) \triangleq \dd A_t F(\dd x)\).
	\begin{lemma}\label{lemma L}
		There exists a PII \(L\) with characteristics \((0, 0, \nu^L)\).
	\end{lemma}
	\begin{proof}
		This follows readily from \cite[Theorem II.5.2]{JS}.
	\end{proof}
	Let \(Z^c\) be a Wiener process with variance function \(C\). We set 
	\begin{align}\label{Z bullet}
	Z^\bullet \triangleq B^{\bullet} + Z^c + h(\rho^\bullet) \star (\mu^L - \nu^L) + (\rho^\bullet - h(\rho^\bullet)) \star \mu^L.
	\end{align}
	\begin{lemma}
		The process \(Z^X\) is well-defined and has the same law as \(X\). Moreover, the process \(Z^Y\) is well-defined and has the same law as \(Y\).
	\end{lemma}
	\begin{proof}
		To establish that \(Z^\bullet\) is well-defined, it suffices to verify that \(\mu^{Z^{\bullet}} ([0, t] \times G) \triangleq \int_0^t \int \1_G (\rho^\bullet (s, x)) \mu^L(\dd s \times \dd x)\), where \(G \in \mathscr{B}(\mathbb{R})\), is a random measure of jumps with compensator \(\nu^{\bullet}\). Since \(\mu^L\) is an optional random measure, so is \(\mu^{Z^\bullet}\). It remains to show that \(\mu^{Z^\bullet}\) is \(\widetilde{\mathscr{P}}\)-\(\sigma\)-finite.
		Consider the set \(G_n \triangleq \{x \in \mathbb{R} \colon |x| > \frac{1}{n}\}\). Now, \(E[\mu^{Z^\bullet}([0, t] \times G_n)] = \nu^{\bullet} ([0, t] \times G_n) < \infty\), see \cite[II.5.5 (i)]{JS}.
		Hence, \(Z^\bullet\) is well-defined. Moreover, it follows readily from \cite[Theorem II.5.10]{JS} that \(Z^\bullet\) is a PII with characteristics \((B^{\bullet}, 0, \nu^{\bullet})\). Therefore, the equality of the laws follows from \cite[Theorem II.5.2]{JS}.
	\end{proof}
	Using \eqref{fv}, we compute that for all \(t \in [0, \infty)\)
	\begin{align*}
	Z^Y_t - Z^X_t &= B^{Y}_t - B^{X}_t - h(x) \star \nu^{Y}_t + h(x) \star \nu^{X}_t + \left(\rho^Y - \rho^X\right) \star \mu^L.
	\end{align*}
	Our assumption \eqref{main coupl st} implies that \(\rho^X\leq \rho^Y\).
	Hence, using \eqref{drift cond new} we obtain \(Z^Y_t \geq Z^X_t\) for all \(t \in [0, \infty)\).
	This concludes the proof of Theorem \ref{theo: Mnew}.
\end{proof}
In view of \cite[Corollary II.5.13]{JS}, the finite variation condition \eqref{fv} and the drift condition \eqref{drift cond new} are independent of the choice of \(h\).

Let us shortly comment on the assumptions of Theorem \ref{theo: Mnew}. 
\begin{proposition}\label{lem: eq cond}
	If \(|x| \star \nu^\bullet_t < \infty\) for all \(t \in [0, \infty)\), then \eqref{main coupl st} holds for \(\dd A_t\)-a.a. \(t \in [0, \infty)\) if, and only if, the order \(K^X\preceq_{st} K^Y\) holds. 
\end{proposition}
\begin{proof}
	Denote \(F, \rho^X\) and \(\rho^Y\) as in the proof of Theorem \ref{theo: Main Uni}, let \(f \in \mathcal{F}^1_{st}\) such that \(|f(x)| \leq \textup{const. } |x|\) and suppose that \eqref{main coupl st} holds.
	Now, recalling \eqref{eq: ito}, for \(\dd A_t\)-a.a. \(t \in [0, \infty)\)
	\begin{align*}
	\int f(x) K^X(t, \dd x) &= \int f(\rho^X(t, x)) F(\dd x) 
	\\&\leq \int f(\rho^Y(t, x)) F(\dd x) = \int f(x) K^Y(t, \dd x).
	\end{align*}
	In other words, \(K^X\preceq_{st}K^Y\) holds.

	For the converse direction, we approximate.
	 Namely, consider \(f(y)\triangleq - \1_{(- \infty, x]}(y)\) for some \(x > 0\) and set \(f_n\) to be the inf-convolution of \(f\), i.e. \(f_n (y) \triangleq \inf_{z \in \mathbb{R}} (f(z) + n |y - z|)\). 
Since \(f\) is a bounded lower semi-continuous function, by \cite[Lemma 1.3.5]{dupuis2011weak} (see also the proof), the inf-convolution \(f_n\) is Lipschitz continuous, \(f_n(y) \leq f_{n+1}(y) \leq f(y)\) and \(f_n \to f\) pointwise as \(n \to \infty\). Moreover, \(f_n\) is increasing as \(f\) is increasing (see \eqref{eq: inf-conv increasing}) and \(f_n(0) = \inf_{z \in \mathbb{R}} (f(z) + n |z|) = 0\) for \(n\geq |x|^{-1}\).
Thus, using the monotone convergence theorem, \(K^X \preceq_{st} K^Y\) implies the first part of \eqref{main coupl st}.
The second part follows in the same manner.
\end{proof}



So far our conditions for the stochastic order \(\preceq_{pst}\) apply to PIIs with discontinuous parts of finite variation. Similarly to Theorem \ref{theo: Approx 1}, we can relax this assumption using the fact that \(\preceq_{pst}\) is closed under weak convergence. Take a sequence \((G_n)_{n \in \mathbb{N}} \subseteq \mathbb{R}\) of Borel sets such that \(\1_{G_n}\) is vanishing in a neighborhood of the origin and \(\bigcup_{n \in \mathbb{N}} G_n \supseteq \mathbb{R}\backslash \{0\}\).
\begin{theorem}\label{prop: M1}
	Let \(h\) be continuous and suppose that for all \(n \in \mathbb{N}\) and \(\dd A_t\)-a.a. \(t \in [0, \infty)\) the condition \eqref{main coupl st} holds with \(K^X\) replaced by \(\1_{G_n} \cdot K^X\) and \(K^Y\) replaced by \(\1_{G_n} \cdot K^Y\), and that for all \(n \in \mathbb{N}\) and \(t \in [0, \infty)\) the inequality \eqref{drift cond trunc} holds.
	Then \(X \preceq_{pst} Y\).
\end{theorem}
\begin{proof}
	Let \(\widehat{Z}^\bullet (n) \triangleq Z^\bullet(n) + Z^c\), where \(Z^\bullet (n)\) and \(Z^c\) are as in Lemma \ref{lem: qlc}. 
	We deduce from \cite[Theorem VII.3.4]{JS} that \(\widehat{Z}^{X}(n)\) convergences in law to \(X\) and \(\widehat{Z}^{Y}(n)\) converges in law to \(Y^{qlc}\) as \(n \to \infty\).
	Now, Theorem \ref{theo: Mnew} yields \(\widehat{Z}^{X}(n) \preceq_{pst} \widehat{Z}^{Y}(n)\) and by \cite[Proposition 3]{kamae1977} we conclude that \(X \preceq_{pst} Y\). 
\end{proof}


\subsection{Cut Criteria for the Monotone Stochastic Order}
In this section we study the case where the frequencies of jumps change once. Together with some integrability conditions, this implies the stochastic order \(\preceq_{pst}\).
We start with a technical observation:
\begin{lemma}\label{lemma comparison measure}
	There exists a \(\sigma\)-finite measure \(\nu\) on \(([0, \infty) \times \mathbb{R}, \mathscr{B}([0, \infty)) \otimes \mathscr{B}(\mathbb{R}))\) such that \(\nu^{X}\ll \nu\) and \(\nu^{Y} \ll \nu\).
\end{lemma}
\begin{proof}
	Since \(\nu^{X}\) and \(\nu^{Y}\) are \(\sigma\)-finite, by the Radon-Nikodym theorem, \(\nu \triangleq \nu^{X} + \nu^{Y}\) has the desired properties.
\end{proof}
Let \(|\) be either \([\) or \(]\) and let \(|^c\) be the converse, i.e. if \(| =\ ]\), then \(|^c = [\). 

\begin{proposition}\label{theo: M1}
	Let \(k \in \mathbb{R}\) and suppose that \(\nu\)-a.e.
	\begin{equation}\label{K cond main}
	\begin{split}
	\frac{\dd \nu^{X}}{\dd \nu}\1_{[0, \infty) \times |k, \infty)} &\leq \frac{\dd \nu^{Y}}{\dd \nu}\1_{[0, \infty) \times |k, \infty)},\\
	\ \frac{\dd \nu^{X}}{\dd \nu}\1_{[0, \infty) \times (- \infty, k|^c}&\geq \frac{\dd \nu^{Y}}{\dd \nu} \1_{[0, \infty) \times (- \infty, k|^c}.
	\end{split}
	\end{equation}
	Moreover, assume for \(\dd A_t\)-a.a. \(t \in [0, \infty)\) 
	\begin{equation}\label{eq: comp 1}
	\begin{split}
	\int_{|k, 0]} \left(K^Y(t, \dd x) - K^X(t, \dd x)\right) \leq &\int_{(- \infty, k|^c}\left(K^X(t, \dd x) - K^Y(t, \dd x)\right)
	\end{split}
	\end{equation}
	in the case \(k < 0\), and 
	\begin{equation}\label{eq: comp 2}
	\begin{split}
	\int_{[0, k|^c} \left(K^X(t, \dd x) - K^Y(t, \dd x)\right) \leq &\int_{| k, \infty)}\left(K^Y(t, \dd x) - K^X(t, \dd x)\right)
	\end{split}
	\end{equation}
	in the case \(k > 0\).
	Finally, suppose that for all \(t \in [0, \infty)\)
	\begin{align}\label{main inte cond}
	\left|h(x)\left(\frac{\dd \nu^{Y}}{\dd \nu} - \frac{\dd \nu^{X}}{\dd \nu}\right)\right| \star \nu_t < \infty,
	\end{align}
	\begin{align}\label{K cond coro}
	B^{Y}_t- B^{X}_t - h(x) \left(\frac{\dd \nu^{Y}}{\dd \nu} - \frac{\dd \nu^{X}}{\dd \nu}\right)\star \nu_t  \geq 0.
	\end{align}
	Then \(X \preceq_{pst} Y\).
\end{proposition}
We stress that the r.h.s. of \eqref{eq: comp 1} and \eqref{eq: comp 2} are finite due to \cite[II.5.5 (i)]{JS}.
\begin{remark}
	Suppose that \(X\) and \(Y\) are semimartingales with independent increments and that their laws are locally absolutely continuous. Then, by Girsanov's theorem \cite[Theorem III.3.24]{JS}, we find \(\nu\) such that \eqref{main inte cond} holds and \eqref{K cond coro} only depends on the Gaussian parts of \(X\) and \(Y\).
\end{remark}

We provide the intuitions behind the assumptions of Theorem \ref{theo: M1}. The condition \eqref{K cond main} means that \(Y\) has a higher frequency of jumps with size larger than \(k\) compared to \(X\) and that \(X\) has a higher frequency of jumps with size less than \(k\) compared to \(Y\).
The conditions \eqref{eq: comp 1} and \eqref{eq: comp 2} compensate negative jumps which are done by \(Y\) but not by \(X\) and positive jumps which are done by \(X\) but not by \(Y\).

Following this intuition, we can construct explicit couplings of \(X\) and \(Y\) which are pathwise ordered. We reveal that Proposition \ref{theo: M1} is (modulo integrability issues) a consequence of Theorem \ref{theo: Mnew}. However, we think the alternative proof explains very nicely the origin of the conditions and illustrates the relations of the characteristics of the PIIs and the stochastic order.
\\\\
\begin{proof}
	We only discuss the cases \(k = 0\) and \(k > 0\), since the case \(k < 0\) follows similarly to the case \(k > 0\).
	
	\emph{The case \(k = 0\).} We set
	\(
	\nu^{X} \wedge \nu^{Y}\triangleq\1_{[0, \infty)}\cdot \nu^{X}+ \1_{(- \infty, 0]}\cdot \nu^{Y}.\)
	Due to \cite[Theorem II.5.2]{JS} we find a filtered probability space which supports a PII \(Z\) with characteristics \((0, C,\nu^{X} \wedge \nu^{Y})\), a PII \(Z^X\) with characteristics \((B^{X}, 0, \1_{(- \infty, 0]} \cdot (\nu^{X}- \nu^{Y}))\) and a PII \(Z^Y\) with characteristics \((B^{Y}, 0, \1_{[0, \infty)}\cdot (\nu^{Y} -\nu^{X}))\) such that \(Z, Z^X\) and \(Z^Y\) are independent. 
	Then, by independence and \cite[Lemma II.2.44, Theorem II.5.2]{JS},
	\(
	Z + Z^X
	\)
	has the same law as \(X\) and
	\(
	Z + Z^Y
	\)
	has the same law as \(Y\).
	Moreover,
	a.s. \(x\1_{[0, \infty)} \star \mu^{Z^X}_t = 0\) and \(x\1_{(-\infty, 0]} \star \mu^{Z^Y}_t = 0\) for all \(t \in [0, \infty)\), noting the support of the third characteristics.
	Thanks to \cite[Proposition II.1.28]{JS} and \eqref{main inte cond}, we obtain for all \(t \in [0, \infty)\)
	\begin{align*}
	Z^Y_t - Z^X_t
	= B^{Y}_t - B^{X}_t &- h(x) \left(\frac{\dd \nu^{Y}}{\dd \nu} - \frac{\dd \nu^{X}}{\dd \nu}\right) \star \nu_t 
	\\&+ x \1_{[0, \infty)}  \star \mu^{Z^Y}_t - x\1_{(- \infty, 0]} \star\mu^{Z^X}_t.
	\end{align*}
	Hence, using \eqref{K cond coro}, \(Z_t + Z^X_t \leq Z_t + Z^Y_t\) for all \(t \in [0, \infty)\). This proves \(X \preceq_{pst} Y\).
	
	\emph{The case \(k > 0\).}  
	We define a Borel measure \(\nu^+\) on \([0, \infty) \times \mathbb{R} \times \mathbb{R} \times \mathbb{R}\) by 
	\begin{align*}
	&\frac{\nu^+ (\dd t \times \dd x \times \dd y \times \dd u)}{\left[ K^X(t, \dd x) - K^Y(t, \dd x)\right] \left[K^Y(t, \dd y) - K^X(t, \dd y)\right]\dd u \dd A_t}
	\\&\hspace{5cm} \triangleq \frac{\1_{[0, k|^c}(x) \1_{|k, \infty)} (y) \1_{[0, 1]}(u) }{\int_{[0, k|^c} (K^X(t, \dd z) - K^Y(t, \dd z))},
	\end{align*}
	with the convention that \(\frac{1}{0} = 1\). Here, we use \eqref{eq: comp 2}, i.e. that the denominator on the r.h.s. is bounded from above for \(\dd A_t\)-a.a. \(t \in [0, \infty)\), see \cite[II.5.5 (i)]{JS}.
	\begin{lemma}
		There exists a PII \(Z^+\) with characteristics \((0, 0, \nu^+)\).
	\end{lemma}
	\begin{proof}
		In view of \cite[Theorem II.5.2]{JS} it suffices to show that \(\nu^+\) is finite on \([0, t] \times \mathbb{R}^3\) for all \(t \in [0, \infty)\). 
		It holds that 
		\(
		\nu^+([0, t] \times \mathbb{R}^3) = (\nu^{Y} - \nu^{X}) ([0, t] \times |k, \infty)) < \infty,
		\)
		see \cite[II.5.5 (i)]{JS}.
		Hence, the lemma is proven.
	\end{proof}
	We find a filtered probability space which supports a PII \(Z\) with characteristics \((0, C, \nu^X \wedge \nu^Y)\), where \(\nu^X \wedge \nu^Y \triangleq \1_{|k, \infty)} \cdot \nu^X + \1_{(- \infty, k|^c} \cdot \nu^Y\), a PII \(Z^X\) with characteristics \((B^X, 0, \1_{(- \infty, 0]} \cdot (\nu^X - \nu^Y))\) and the PII \(Z^+\) such that they are independent.
	We define 
	\begin{align*}m(t, u) \triangleq \1 \left\{ u \leq \frac{\int_{[0, k|^c} (K^X(t, \dd z) - K^Y(t, \dd z))}{\int_{|k, \infty)} (K^Y(t, \dd z) - K^X(t, \dd z))}\right\}.
	\end{align*}
	Now, set
	\begin{align*}
	Z^{X, -} &\triangleq x m(t, u) \1_{[0, k|^c}(x) \1_{|k, \infty)}(y)\star \mu^{Z^+} 
	- h \1_{[0, k|^c} \star (\nu^{X} - \nu^{Y}),\\
	Z^{Y} &\triangleq y \1_{[0, k|^c}(x) \1_{|k, \infty)}(y)\star \mu^{Z^+} + B^{Y}
	- h \1_{|k, \infty)} \star (\nu^{Y} - \nu^{X}).
	\end{align*}
	By assumption \eqref{main inte cond}, the processes \(Z^{X, -}\) and \(Z^Y\) are well-defined.
	\begin{lemma}
		The process \(Z^{X, -}\) is a PII with characteristics \((0, 0, \1_{[0, k|^c} \cdot (\nu^{X}- \nu^{Y}))\) and \(Z^Y\) is a PII with characteristics \((B^{Y}, 0, \1_{|k, \infty)} \cdot (\nu^{Y} - \nu^{X}))\).
	\end{lemma}
	\begin{proof}
		Let \(\mathscr{C}^+\) be a set of test functions as defined in \cite[II.2.20]{JS}. All functions in \(\mathscr{C}^+\) are bounded and vanish in a neighborhood of the origin.
		In view of \cite[Proposition I.1.28, Theorem II.1.33, Theorem II.5.10]{JS} for the first claim suffices to show that for all \(g \in \mathscr{C}^+\) the process 
		\(
		g \star \mu^{Z^{X, -}} - g  \1_{[0, k|^c} \star (\nu^{X}- \nu^{Y})
		\)
		is a local martingale. In fact, since \(g\) vanishes in a neighborhood of the origin, we have
		\(
		g \star \mu^{Z^{X, -}} = g (x) m(t, u) \1_{[0, k|^c} (x) \1_{|k, \infty)}(y) \star \mu^{Z^+},
		\)
		which compensator is given by
		\(
		g (x) m(t, u) \1_{[0, k|^c} (x) \1_{|k, \infty)}(y) \star \nu^{+} 
		= g \1_{[0, k|^c} \star \left(\nu^X - \nu^Y\right).
		\)
		Here, we use \eqref{eq: comp 2}.
		This proves the first claim. The second claim follows in the same manner.
	\end{proof}
	Now, by independence and  \cite[Lemma II.2.44, Theorem II.5.2]{JS}, \(Z + Z^X + Z^{X,-}\) has the same law as \(X\) and \(Z + Z^Y\) has the same law as \(Y\). Moreover, for all \(t \in [0, \infty)\)
	\begin{align*}
	Z^Y_t - Z^X_t - Z^{X, -}_t = B^{Y}_t &- B^{X}_t - h(x) \left(\frac{\dd \nu^{Y}}{\dd \nu} - \frac{\dd \nu^{X}}{\dd \nu}\right) \star \nu_t + x \1_{[0, \infty)}  \star \mu^{Z^Y}_t
	\\&+ (y - x m(t, u))\1_{[0, k|^c}(x)\1_{|k, \infty)}(y) \star \mu^{Z^+}_t.
	\end{align*}
	Since \(m(t, u) \leq 1\), we obtain \(Z^Y_t - Z^X_t - Z^{X, -}_t \geq 0\) for all \(t \in [0, \infty)\). Again, this proves \(X \preceq_{pst} Y\).
\end{proof}

\subsection{A Majorization Condition for the Convex Stochastic Order}\label{sec: C}
In this section we show that if the third characteristic of \(Y\) dominates the third characteristics of \(X\), then \(X \preceq_{(i)cx} Y\). Our proof is based on a coupling. We stress that the proof is very robust in sense that it applies to all PIIs with finite first moments.
\begin{theorem}\label{theo: C}
	Suppose that for all \(t \in [0, \infty)\) the matrix \(C^Y_t - C^X_t\) is non-negative definite, \(\nu^{Y}- \nu^{X}\) is a non-negative measure, \(|x - h(x)| \star \nu^\bullet_t < \infty\) and 
	\begin{align}\label{eq: ineq moment}
	B^X_t + |x - h(x)| \star \nu^X_t \leq B^Y_t + |x - h(x)| \star\nu^Y_t,
	\end{align}
	Then \(X \preceq_{icx} Y\). If, additionally, \eqref{eq: ineq moment} holds with equality, then also \(X \preceq_{cx} Y\). 
\end{theorem}

Theorem \ref{theo: C} generalizes the majorization criterion \cite[Proposition 19]{bergenthum2007b} for L\'evy processes with infinite activity by showing that the conditions \cite[Proposition 19 (a), (b)]{bergenthum2007b} are not necessary.
\\

\begin{proof}
	In view of \cite[Theorem II.5.2]{JS}, we can extend our probability space such that it supports a PII \(Z^{Y}\) with characteristics \((B^{Y} - B^{X}, C^Y - C^X, \nu^Y - \nu^X)\), which is independent of \(X\). 
	Now, by independence and \cite[Lemma II.2.44, Theorem II.5.2]{JS}, the process \(X + Z^{Y}\) has the same law as \(Y\).
	Note that 
	\(
	E \left[Z^{Y}_t\right] = B^{Y}_t - B^{X}_t + (x - h(x)) \star (\nu^{Y} -\nu^{X})_t
	\) for all \(t \in [0, \infty)\).
	Hence, we obtain a.s. for all \(t \in [0, \infty)\)
	\begin{align*}
	E \left[X_t + Z^Y_t |\sigma(X_s, s \in [0, \infty))\right] &= X_t + E \left[Z^Y_t\right]
	\\&\begin{cases}
	\geq X_t,&\textup{ if \eqref{eq: ineq moment} holds}, \\
	= X_t,&\textup{ if \eqref{eq: ineq moment} holds with equality}.
	\end{cases}
	\end{align*}
	Now, \(X \preceq_{(i)cx} Y\) follows from the conditional Jensen's inequality.
\end{proof}

\section{Generalizations and Examples}\label{sec: appl}
In this section we discuss generalizations of our conditions to semimartingales with conditionally independent increments (called \(\mathscr{H}\)-SIIs in the following). 
Moreover, we give examples.
\subsection{Comparison of \(\mathscr{H}\)-SIIs}\label{Comparison of Semimartingales with Conditionally Independent Increments}
Let \((\Omega, \mathscr{F})\) be a Polish space equipped with its topological Borel \(\sigma\)-field and let \(\p\) be a probability measure on \((\Omega, \mathscr{F})\).
We choose two not necessarily right-continuous filtrations \((\mathscr{F}^1_t)_{t \in [0, \infty)}\) and \((\mathscr{F}^2_t)_{t \in [0, \infty)}\) on \((\Omega, \mathscr{F})\) consisting of countably generated \(\sigma\)-fields. 
Moreover, let \(\mathscr{H}^1\) and \(\mathscr{H}^2\) be two countably generated sub-\(\sigma\)-fields of \(\mathscr{F}\).
For example, \(\mathscr{H}^i\) could be \(\sigma(Z_s, s \in [0, \infty))\) where \(Z\) is right- or left-continuous and takes values in a Polish space. 
For \(i =1, 2\) we define the filtrations \(\G^i = (\mathscr{G}^i_t)_{t \in [0, \infty)}\) 
on \((\Omega, \mathscr{F})\)~by 
\(
\mathscr{G}^i_t \triangleq \mathscr{G}^{i, o}_{t+},\) where
\(\mathscr{G}^{i, o}_t \triangleq \mathscr{F}^i_t \vee \mathscr{H}^i. 
\)

A process \(B \in \mathscr{V}^d\) (or \(\in \mathscr{V}^{d \times d}\)) is said to have an \(\mathscr{H}^i\)-measurable version, if for all \(t \in [0, \infty)\) the random variable \(B_t\) has an \(\mathscr{H}^i\)-measurable version.
We say that a compensator \(\nu\) of a random measure of jumps has an \(\mathscr{H}^i\)-measurable version, if for all \(t \in [0, \infty)\) and all Borel functions \(g \colon \mathbb{R}^d \to \mathbb{R}\) such that \(|g(x)| \leq 1 \wedge |x|^2\) the random variable \(g \star \nu_t\) has an \(\mathscr{H}^i\)-measurable version.


Let \(X\) be an \(\mathbb{R}^d\)-valued \(\G^1\)-semimartingale with characteristics \((B^X, C^X, \nu^X)\), and \(Y\) be an \(\mathbb{R}^d\)-valued \(\G^2\)-semimartingale with characteristics \((B^Y, C^Y, \nu^Y)\), such that \((B^X, C^X, \nu^X)\) has an \(\mathscr{H}^1\)-measurable version and that \((B^Y, C^Y, \nu^Y)\) has an \(\mathscr{H}^2\)-measurable version.

In this setting there exists a regular conditional probability \(\p(\cdot|\mathscr{H}^i)(\cdot)\) from \((\Omega, \mathscr{H}^i)\) to \((\Omega, \mathscr{F})\), see, e.g., \cite[Theorem 9.2.1]{stroock2010}. 
The following  is the main observation to transfer our conditions for PIIs to \(\mathscr{H}\)-SIIs.
\begin{lemma}\label{SII lemma}
	There exists a null set \(N \in \mathscr{F}\) such that for all \(\omega \in \complement N\) the process
	\(X\) is a \(P(\cdot|\mathscr{H}^1)(\omega)\)-PII with characteristics \((B^X(\omega), C^X(\omega), \nu^X(\omega))\) and \(Y\) is a \(P(\cdot|\mathscr{H}^2)(\omega)\)-PII with characteristics \((B^Y(\omega),\) \(C^Y(\omega), \nu^Y(\omega))\).
\end{lemma}
\begin{proof}
	The claim follows from \cite[Lemma II.6.13, Corollary II.6.15]{JS} and \cite[Lemma A.1]{C15b}. 
\end{proof}
In view of this lemma it is strait forward to transfer our conditions for PIIs to \(\mathscr{H}\)-SIIs. More precisely, assume that for all \(\omega \in \complement N\), where \(N\) is as in the previous lemma, the characteristics \((B^X(\omega), C^X(\omega), \nu^X(\omega))\) and \((B^Y(\omega), C^Y(\omega), \nu^Y(\omega))\) satisfy our conditions for \(\preceq_{\bullet}\). Then, with abuse of notation, for all \(f \in \mathcal{F}_{\bullet}\)
\[
\int f(X(\omega^*)) P(\dd \omega^*|\mathscr{H}^1) (\omega) \leq \int f(Y(\omega^*)) P(\dd \omega^*|\mathscr{H}^2)(\omega).
\]
Since \(N\) is a null set, taking expectation yields \(X \preceq_\bullet Y\).

We omit to state explicit conditions as these are similar to our previous conditions with an additional almost surely. However, we give examples in the following section.

\subsection{Examples}\label{Comparison of Levy processes}
In this section we provide examples. Here, we focus on processes which are not discussed in \cite{bergenthum2007,bergenthum2007b}.
We start with an example  of a PII with infinite variation and fixed times of discontinuity.

\begin{example}[Time-inhomogeneous CGMY L\'evy processes with fixed times of discontinuity]\label{levy example}
	For \(n \in \mathbb{N}\) fix \(0 < t_1 < t_2 < ... < t_n < \infty\). Moreover, for \(k =1 , 2\) and \(i = 1, ..., n\) let \(F^{k, i}\) be probability measures on \((\mathbb{R}, \mathscr{B}(\mathbb{R}))\) such that \(F^{k, i}(\{0\}) = 0\) and for all \(x \in \mathbb{R}\) we have 
	\(
	F^{2, i} ((- \infty, x]) \leq F^{1, i} ((-\infty, x]).
	\)
	Let \(X^1\) and \(X^2\) be real-valued PIIs such that
	\(C^{X^1} = C^{X^2} = 0\) and 
	\begin{align*}
	\nu^{X^1}(\dd t \times \dd x) &= \frac{C_t}{|x|^{1+ Y_t}} e^{- M^1_t x} \1_{\{x > 0\}} \dd t \dd x \\&\hspace{1cm}+ \frac{C_t}{|x|^{1 + Y_t}} e^{G^1_t x} \1_{\{x < 0\}} \dd t \dd x + \sum_{i = 1}^n \varepsilon_{t_i}(\dd t) F^{1, i}(\dd x),\\ 
	\nu^{X^2}(\dd t \times \dd x) &= \frac{C_t}{|x|^{1 + Y_t}} e^{- M^2_t x} \1_{\{x > 0\}} \dd t\dd x \\&\hspace{1cm}+ \frac{C_t}{|x|^{1 + Y_t}} e^{G^2_t x}\1_{\{x < 0\}} \dd t\dd x + \sum_{i = 1}^n \varepsilon_{t_i}(\dd t) F^{2, i}(\dd x).
	\end{align*}
	Here, 
	\(C, G^1, G^2, M^1, M^2 \colon [0, \infty) \to (0, \infty)\) are continuous and bounded away from \(0\) on compact sets, and \(Y \colon [0, \infty) \to (- \infty, 2)\) is continuous and bounded away from \(2\) on compact sets. 
	If \(\nu= \nu^X + \nu^Y\), then the integrability condition \eqref{main inte cond} holds.
	Moreover, \eqref{K cond main} is satisfied if 
	\(
	M^2_t \leq M^1_t\) and \(G^1_t \leq G^2_t\) for all \(t \in [0, \infty)\) except of a \(\dd t\)-null set.
	If we choose \(B^{X^1}\) and \(B^{X^2}\) according to \eqref{K cond coro}, the Propositions \ref{prop: fts qlc ordnen} and \ref{theo: M1}
	imply~\(X^1\preceq_{pst} X^2\).
\end{example}

Next, we compare \(\mathscr{H}\)-SIIs, using the notation from the previous section.
\begin{example}[Comparison of Time-Changed L\'evy Processes]
	For \(i = 1, 2\), let \(Z^i\) be a \([0, \infty)\)-valued \cadlag processes and let \(V^i\) be a real-valued L\'evy process (w.r.t. its natural filtration) with L\'evy-Khinchine triplet \((b^{V^i}, 0, F^{V^i})\). 
	Assume that the process \(Z^i\) is independent of \(V^i\) 
	and set
	\(
	X^i_\cdot \triangleq V^i_{\int_0^\cdot Z^i_s \dd s},
	\)
	\(
	\mathscr{F}^i_t \triangleq \sigma(X^i_s, Z^i_s, s \in [0, t])\) and \(\mathscr{H}^i  \triangleq \sigma(Z^i_t, t \in [0, \infty)).\)
	The next lemma follows from \textup{\cite[Lemma 2.4]{KMK10}}.
	\begin{lemma}
		For \(i = 1,2\) the process \(X^i\) is a \(\G^i\)-semimartingale and its \(\G^i\)-semimartingale characteristics are given by
		\begin{align*}
		B^{X^i} = b^{V^i}\int_0^\cdot Z^i_s \dd s,\quad C^{X^i} = 0,\quad \nu^{X^i}(\dd t \times \dd x) = Z^i_{t-} \dd t  F^{V^i}(\dd x).
		\end{align*}
	\end{lemma}
	Using Lemma \ref{SII lemma} and the Theorems \ref{theo: Mnew} and \ref{theo: C}, we obtain the following
	\begin{corollary}\begin{enumerate}
			\item[\textup{(a)}]
			Suppose that  \(\int|h(x)| F^{V^i}(\dd x) < \infty\), that
			\begin{align*}
			F^{V^2}((- \infty, x]) &\leq F^{V^1}((- \infty, x])\ x < 0,\\ F^{V^1}([x, \infty)) &\leq F^{V^2}([x, \infty)),\ x > 0,
			\end{align*}
			and that a.s. for all \(t \in [0, \infty)\) 
			\[\hspace{-1cm}\left(b^{V^1} - \int h(x) F^{V^1}(\dd x)\right) \int_0^t Z^1_s \dd s \leq \left(b^{V^2} - \int h(x) F^{V^2}(\dd x)\right) \int_0^t Z^2_s \dd s.\]
			Then \(X \preceq_{pst} Y\).
			\item[\textup{(b)}]
			If a.s. \(Z^1_t \leq Z^2_t\) for all \(t \in [0, \infty)\), \(F^{V^2} - F^{V^1}\) is a non-negative measure,
			\(
			\int|x - h(x)|  F^{V^i} (\dd x) < \infty,
			\)
			and 
			\begin{align}\label{eq:extc}
			b^{V^1} \leq\int (x- h(x)) F^{V^1}(\dd x),\quad b^{V^2} \geq \int(x - h(x)) F^{V^2}(\dd x), 
			\end{align}
			then \(X^1 \preceq_{icx} X^2\). If, additionally, the inequalities in \eqref{eq:extc} are equalities, then \(X^1 \preceq_{cx} X^2\).
		\end{enumerate}
	\end{corollary}
\end{example}
\begin{example}[Comparison of Integrated L\'evy Processes]
	For \(i = 1, 2\), let \(B^i\) be a \(d\)-dimensional Brownian motion (w.r.t. its natural filtration), and let \(\mu^i\) be a left-continuous \(\mathbb{R}^d\)-valued process and \(\sigma^i\) be a left-continuous \(\mathbb{R}^d \otimes \mathbb{R}^d\)-valued process, both of which are independent of \(B^i\). Moreover, we set 
	\(
	\mathscr{F}^i_t \triangleq \sigma(B^i_s, \mu^i_s, \sigma^i_s, s \in [0, t])\) and \(\mathscr{H}^i \triangleq \sigma(\mu^i_t, \sigma^i_t, t \in [0, \infty)).
	\)
	Note that, since \(B^i\) is independent of both \(\mu^i\) and \(\sigma^i\), it follows from \cite[Theorem II.68.2]{RW} and \cite[Theorem 15.5]{bauer2002wahrscheinlichkeitstheorie} that \(B^i\) is an \((\mathscr{F}_{t+}^i)_{t \in [0, \infty)}\)-Brownian motion. Thus, the process
	\[
	X^i \triangleq \int_0^\cdot \mu^i_s \dd s + \int_0^\cdot \sigma^i_s \dd B^i_s
	\]
	is well-defined as an \((\mathscr{F}^i_{t+})_{t \in [0, \infty)}\)-semimartingale.
	The following lemma follows from \textup{\cite[Lemma 2.3]{KMK10}}.
	\begin{lemma}
		For \(i = 1, 2\) the process \(X^i\) is a \(\mathsf{G}^i\)-semimartingale and its \(\G^i\)-semimartingale characteristics are given by 
		\[
		B^{X^i} = \int_0^\cdot \mu^i_s \dd s,\quad C^{X^i} = \int_0^\cdot \sigma^i_s (\sigma^i_s)^* \dd s, \quad \nu^{X^i} = 0,
		\]
		where \((\sigma^i)^*\) denotes the adjoint of \(\sigma^i\).
	\end{lemma}
	Obviously, the \(\G^i\)-semimartingale characteristics of \(X^i\) are \(\mathscr{H}^i\)-measurable.
	We deduce the following comparison result from Lemma \ref{SII lemma} and Theorem \ref{theo: C}.
	\begin{corollary}\label{coro:contsmcx}
		If a.s. 
		\(
		\mu^1_t \leq \mu^2_t\) and \(\sigma^2_t (\sigma^2_t)^* - \sigma^1_t (\sigma^1_t)^*\) is a non-negative definite matrix for all \(t \in [0, \infty)\),
		then \(X \preceq_{icx} Y\). If, additionally, a.s.
		\(
		\mu^1_t = \mu^2_t
		\) for all \(t \in [0, \infty)\),
		then \(X \preceq_{cx} Y\). 
	\end{corollary}
	\begin{remark}\label{rem:SCII cont}
		Results in the spirit of Corollary \ref{coro:contsmcx} were given by Hajek~\cite{Hajek1985} using a coupling technique different from ours.
		More precisely, in the one-dimensional case, Hajek shows the first part of Corollary \ref{coro:contsmcx} in the case where \(\mu^2\) and \(\sigma^2\) are constant and \(\mu^1\) and \(\sigma^1\) are not necessarily independent of \(B^1\).
	\end{remark}
\end{example}

\bibliographystyle{
chicago}
\bibliography{References}

\end{document}